\documentclass[10pt]{article}
\font\smallit=cmti10
\font\smalltt=cmtt10

\usepackage{amssymb,latexsym,amsmath,epsfig,amsthm,mathrsfs} %% Add other packages as necessary

\numberwithin{equation}{section}
\theoremstyle{plain}

\makeatletter

\renewcommand\section{\@startsection {section}{1}{\z@}
{-30pt \@plus -1ex \@minus -.2ex}
{2.3ex \@plus.2ex}
{\normalfont\normalsize\bfseries}}

\renewcommand\subsection{\@startsection{subsection}{2}{\z@}
{-3.25ex\@plus -1ex \@minus -.2ex}
{1.5ex \@plus .2ex}
{\normalfont\normalsize\bfseries}}

\renewcommand{\@seccntformat}[1]{\csname the#1\endcsname. }

\newtheorem{theorem}{Theorem}
\newtheorem{lemma}[theorem]{Lemma}

\newtheorem*{conjecture}{\bf Conjecture}

\newcommand{\Z}{\mathbb{Z}}
\newcommand{\eps}{\epsilon}

\newcommand{\CE}{\mathcal{CE}}
\newcommand{\CO}{\mathcal{CO}}
\newcommand{\spmod}{\hspace{-8pt}\pmod}

\DeclareMathOperator{\ce}{ce}
\DeclareMathOperator{\co}{co}

\theoremstyle{remark}

\newtheorem*{remark}{Remark}

\numberwithin{theorem}{section} \numberwithin{equation}{section}

\makeatother

\begin{document}
\begin{center}
\uppercase{\bf Congruences of Concave Composition Functions}
\vskip 20pt
{\bf Keenan Monks}\\
{\smallit Harvard University, Cambridge, MA 02138, USA}\\
{\tt monks@college.harvard.edu}\\
\vskip 10pt
{\bf Lynnelle Ye}\\
{\smallit Stanford University, Stanford, CA 94309, USA}\\
{\tt lynnelle@stanford.edu}\\ 
\end{center}
\vskip 30pt
\centerline{\smallit Received: , Revised: , Accepted: , Published: } % We will fill in the dates
\vskip 30pt

\centerline{\bf Abstract}
\noindent
Concave compositions are ordered partitions whose parts are decreasing towards a central part. We study the distribution modulo $a$ of the number of concave compositions. Let $c(n)$ be the number of concave compositions of $n$ having even length. It is easy to see that $c(n)$ is even for all $n\geq1$. Refining this fact, we prove that $$\#\{n<X:c(n)\equiv 0\pmod 4\}\gg\sqrt{X}$$ and also that for every $a>2$ and at least two distinct values of $r\in\{0,1,\dotsc,a-1\}$, 
$$\#\{n<X: c(n)\equiv r\spmod{a}\} > \frac{\log_2\log_3 X}{a}.$$ We obtain similar results for concave compositions of odd length.

\pagestyle{myheadings} 
\markright{\smalltt INTEGERS: 13 (2013)\hfill} 
\thispagestyle{empty} 
\baselineskip=12.875pt 
\vskip 30pt

\section{Introduction and statement of results}

In their 1967 paper, Parkin and Shanks \cite{parkin} conjectured that the partition function $p(n)$ takes on even and odd values  with equal likelihood. Very little is known about the distribution of the parity of $p(n)$. Recently, Ahlgren \cite{ahlgren} (see also the work of Berndt, Yee, and Zaharescu \cite{Berndt} as well as the works referenced therein) proved that the number of integers with an even number of partitions less than $X$ is on the order of $\sqrt{X}$. This improved on work of Mirsky \cite{mirsky} showing that about $\log\log X$ numbers less than $X$ have partition values in some nonzero residue class modulo any integer $a$, for the special case $a=2$. 

A combinatorial object similar to the set of partitions is the set of concave compositions --- ordered partitions whose summands decrease towards a center summand. We break up these compositions into the following three types, as defined by Andrews in~\cite{andrews}. 

\emph{Concave compositions of even length} are ordered partitions of the form $a_1+a_2+\cdots+a_m+b_1+b_2+\cdots+b_m$ where $$a_1>a_2>\cdots>a_m=b_m<b_{m-1}<\cdots<b_1$$ and $a_m\geq 0$. We denote the number of concave compositions of even length of an integer $n$ by $\ce(n)$.

\emph{Concave compositions of odd length of type $1$} are ordered partitions of the form $a_1+a_2+\cdots+a_{m+1}+b_1+b_2+\cdots+b_m$ where $$a_1>a_2>\cdots>a_{m+1}<b_m<b_{m-1}<\cdots<b_1$$ and $a_{m+1}\geq 0$. We denote the number of concave compositions of odd length of type $1$ of $n$ by $\co_1(n)$.

Finally, \emph{concave compositions of odd length of type $2$} are ordered partitions of the form $a_1+a_2+\cdots+a_{m+1}+b_1+b_2+\cdots+b_m$ where $$a_1>a_2>\cdots>a_{m+1}\leq b_m<b_{m-1}<\cdots<b_1$$ and $a_{m+1}\geq 0$. We denote the number of concave compositions of odd length of type $2$ of $n$ by $\co_2(n)$. Note that all concave compositions of odd length of type $1$ are also of type $2$.

It is natural to consider the distribution of these functions modulo $a$. To this end, we define $$E_f(r,a;X)=\#\{n<X:f(n)\equiv r\spmod a\}.$$ Since the function $\co_1(n)$ is odd exactly when $n$ is a triangular number (see the remark following Lemma~\ref{expansions}), we define the function $\co_1'(n)$ by subtracting $1$ from $\co_1(n)$ if $n$ is triangular and keeping it the same otherwise. Then we have the following theorem.

\begin{theorem}\label{main}\ \\
The following are true.\\
(i) There exists an explicit constant $c>0$ such that for sufficiently large $X$ we have
$$E_{\ce}(0,4;X)>c\sqrt{X}.$$
(ii) There exists an explicit constant $c>0$ and $0<\alpha<1$ such that for sufficiently large $X$ we have
$$E_{\co_1'}(0,4;X)>\frac{X\log^{\alpha}X-cX}{(\sqrt{24X+1}+1)\log^{\alpha}X}.$$ %where $\co_1'\left(\frac{n^2-n}{2}\right)=\co_1\left(\frac{n^2-n}{2}\right)-1$ and $\co_1'(n)=\co_1(n)$ for $n$ not triangular.\\
(iii) There exists an explicit constant $c>0$ such that for sufficiently large $X$ we have
$$E_{\co_2}(0,2;X)>c\sqrt{X}.$$
\end{theorem}
%\begin{remark}
%Since the function $\co_1(n)$ is odd exactly when $n$ is a triangular number (see the remark following Lemma~\ref{expansions}), we define the function $\co_1'(n)$ by subtracting $1$ from $\co_1(n)$ if $n$ is triangular and keeping it the same otherwise.
%\end{remark}

If we consider the more general case of a modulus $a$, we get a result similar to that of Mirsky \cite{mirsky}.

\begin{theorem}\label{loglog}
For every $a>2$ and at least two distinct values of $r$ among $\{0,1,\dotsc,a-1\}$, we have
\begin{align}
E_{\ce}(r,a;X) &> \frac{\log_2\log_3 X}{a},\\
E_{\co_1}(r,a;X) &> \frac{\log_2\log_3 X}{a},\text{ and}\\
E_{\co_2}(r,a;X) &> \frac{\log_2\log_3 X}{a}
\end{align}
for $X$ sufficiently large. In the cases $\co_1$ and $\co_2$ this also applies when $a=2$.
\end{theorem}

Although the above bound is the best we can prove, we expect the true distributions to be much more balanced.

\begin{conjecture}
For any modulus $a\ge2$, we have
\begin{align*}
E_{\co_1}(r,a;X)&\sim \frac{X}{a},\ a\text{ odd}\\
E_{\co_1}(r,a;X) &\sim \frac{2X}{a},\ a\text{ even},\ r\text{ even}\\
E_{\co_1}(r,a;X)&\sim \frac{2\sqrt{2X}}{a},\ a\text{ even},\ r\text{ odd}.
\end{align*}
For $\ce$, we expect similar asymptotics to hold, but without the last case. For $\co_2$, we expect uniformity across residue classes for any $a$.
\end{conjecture}

\section{Proofs}
\subsection{Generating Functions}
When faced with combinatorial objects such as concave compositions, it is natural to consider the generating functions for each object. Andrews found $q$-series expansions for each type of concave composition in Theorems 1-3 of~\cite{andrews}, which we restate here.

\begin{lemma}\label{expansions}
Define the generating functions $\CE(q)=\sum\limits_{n=0}^{\infty}\ce(n)q^n$, $\CO_1(q)=\sum\limits_{n=0}^{\infty}\co_1(n)q^n$, and $\CO_2(q)=\sum\limits_{n=0}^{\infty}\co_2(n)q^n.$ Then we have
$$\CE(q)=\frac{1+\sum\limits_{n=1}^{\infty}\left(-q^{\frac{3n^2-n}{2}}+q^{\frac{3n^2+n}{2}}\right)}{1+\sum\limits_{n=1}^{\infty}(-1)^n\left(q^{\frac{3n^2-n}{2}}+q^{\frac{3n^2+n}{2}}\right)} =1+2q^2+2q^3+4q^4+4q^5+\cdots,$$
$$\CO_1(q)=\frac{1+\sum\limits_{n=1}^{\infty}\left(-q^{6n^2-2n}+q^{6n^2+2n}\right)}{1+\sum\limits_{n=1}^{\infty}(-1)^n\left(q^{\frac{3n^2-n}{2}}+q^{\frac{3n^2+n}{2}}\right)} =1+q+2q^2+3q^3+4q^4+6q^5+\cdots,$$
$$\CO_2(q)=\frac{1+\sum\limits_{n=1}^{\infty}\left(q^{6n^2-8n+3}-q^{6n^2-4n+1}\right)}{1+\sum\limits_{n=1}^{\infty}(-1)^n\left(q^{\frac{3n^2-n}{2}}+q^{\frac{3n^2+n}{2}}\right)} =1+2q+3q^2+4q^3+7q^4+10q^5+\dotsb.$$
\end{lemma}

\begin{remark}
It is easy to show that $\CO_1(q)\equiv\sum_{n=0}^{\infty}q^{n(n+1)/2}\pmod{2}$, as follows. We have by Lemma~\ref{expansions} that
\[
\CO_1(q)=\frac{1-\sum_{n=1}^{\infty}\left(q^{4n(3n-1)/2}-q^{4n(3n+1)/2}\right)}{1+\sum_{n=1}^{\infty}(-1)^n\left(q^{n(3n-1)/2}+q^{n(3n+1)/2}\right)}
\equiv(q)_{\infty}^3\spmod{2}
\]
where $(q)_{\infty}=\prod\limits_{n=1}^{\infty}(1-q^n).$ By Lemma 12 of~\cite{andrews}, we have that
\[
\sum_{n=0}^{\infty}(-1)^nq^{n(n+1)/2}=(q)_{\infty}\sum_{n=0}^{\infty}\frac{q^{n(n+1)}}{(q)_n^2}\equiv
(q)_{\infty}\left(\sum_{n=0}^{\infty}\frac{q^{n(n+1)/2}}{(q)_n}\right)^2\equiv(q)_{\infty}^3\spmod{2}
\]
since the partitions into distinct parts are conjugate to the partitions into $1,2,\dotsc,n$ missing nothing. The desired congruence follows.

Motivated by this congruence, we define $\CO_1'(q)=\sum\limits_{n=0}^{\infty}\co_1'(n)q^n=\CO_1(q)-\sum_{n=0}^{\infty}q^{n(n+1)/2}$, so that all of the coefficients of $\CO_1'(q)$ are even. This will ease our study of the series modulo $4$. With these, we can now proceed to the proofs of our main theorems.
\end{remark}

\subsection{Proofs of Theorems~\ref{main} and \ref{loglog}}
Combining the expansions given by Andrews \cite{andrews} with a generalization of Ahlgren's argument in \cite{ahlgren}, we prove our main theorem.
\begin{proof}[Proof of Theorem~\ref{main}]

(i) First, we show that the coefficients $\ce(n)$ are all even for $n>1$.  There is a natural pairing between non-palindromic concave compositions of even length given by mirroring the sequence. There is furthermore a natural pairing between palindromic compositions given by inserting or removing a pair of zeroes at the center of the sequence. Alternatively, Lemma~\ref{expansions} makes it clear that $\CE(q)\equiv1\pmod2$.

Thus we have that $\ce(n)$ is congruent to either $0$ or $2 \bmod 4$. From Lemma~\ref{expansions}, we have the following:
$$\left(1+\sum\limits_{n=1}^{\infty}(-1)^n\left(q^{\frac{3n^2-n}{2}}+q^{\frac{3n^2+n}{2}}\right)\right)\left(\sum_{n=0}^{\infty}\ce(n)q^n\right)=1+\sum\limits_{n=1}^{\infty}\left(-q^{\frac{3n^2-n}{2}}+q^{\frac{3n^2+n}{2}}\right).$$ 

We will denote the right hand side by $\sum\limits_{n=0}^{\infty}a(n)q^n$. Thus, since $-2\equiv 2\pmod 4$, we have that 
\begin{equation}\label{decomposition}a(n)\equiv\ce(n)+\ce(n-1)+\ce(n-2)+\ce(n-5)+\cdots+\ce\left(n-\frac{3k^2\pm k}{2}\right)+\cdots\spmod 4.\end{equation}

If $a(n)$ is divisible by $4$ and there are an odd number of summands, we can conclude that one of the summands is also divisible by $4$. It is easy to see that there will be an odd number of summands exactly when $\frac{3k^2+k}{2}<n<\frac{3(k+1)^2-(k+1)}{2}.$ As $X$ tends to infinity, it is easy to show that the number of such $n<X$ tends to $\frac{2}{3}X$ from below very quickly, since $\frac{3(k+1)^2-(k+1)}{2}$ is almost exactly two-thirds of the way from $\frac{3k^2+k}{2}$ to $\frac{3(k+1)^2+(k+1)}{2}$. 

Thus for approximately $\frac{2}{3}X$ values of $n$, one of the terms $\ce(i)$ must be congruent to $0$ modulo $4$. These terms may be overcounted by the number of decompositions of the form (\ref{decomposition}) in which they appear. This is bounded above by twice the number of pentagonal numbers less than $X$, which is $\frac{\sqrt{24X+1}+1}{3}$. Thus we can conclude that we have, for some small constant $\eps$,
$$E_{\ce}(0,4;X)>\frac{\left(2-\eps\right)X}{\sqrt{24X+1}+1},$$ as desired.

(ii) We first write out the expansion of $\CO_1'$ as in (i): 
\begin{align*}
&\left(1+\sum\limits_{n=1}^{\infty}(-1)^n\left(q^{\frac{3n^2-n}{2}}+q^{\frac{3n^2+n}{2}}\right)\right)\left(\sum_{n=0}^{\infty}\co_1'(n)q^n\right)\\
= &1+\sum\limits_{n=1}^{\infty}\left(-q^{6n^2-2n}+q^{6n^2+2n}\right)-\sum_{j,k=1}^{\infty}\left(q^{\frac{j^2-j+3k^2-k}{2}}+q^{\frac{j^2-j+3k^2+k}{2}}\right).
\end{align*}
Again writing the right hand side as $\sum\limits_{n=0}^{\infty}a(n)q^n$, we notice that $a(n)=0$ whenever $n$ is not expressible as $6k^2\pm 2k$ or as the sum of a triangular and a pentagonal number. To obtain a bound on how many such terms there are, we first notice that there are at most $\frac{\sqrt{6X+1}-1}{3}$ values of $n<X$ expressible as $6k^2\pm2k$. To find how many numbers less than $X$ are expressible as the sum of a triangular and a pentagonal number, we use the result that if $Q(x,y)$ is a positive definite binary quadratic form,
$$\#\{n<X: n=Q(x,y)\text{ for some }x,y\in\Z\}\asymp\frac{X}{\log^{\alpha}X}$$
 for some $0<\alpha<1$. For a presentation of a similar result, see Section 2 in \cite{Serre}.

Thus we have that there are at least $X\left(\frac{\log^{\alpha}X-c}{\log^{\alpha}X}\right)$ numbers $n<X$ such that $a(n)=0$. Then by an argument analogous to that in (i), we can conclude that 
$$E_{\co_1'}(0,4;X)>\frac{X\log^{\alpha}X-cX}{(\sqrt{24X+1}+1)\log^{\alpha}X}.$$

(iii) This proof is analogous to the proof of (i), the only difference being that we need to exclude the values of $n$ for which $a(n)$ is nonzero. Thus the bound we get in this case is, for some small constant $\eps$, 
$$E_{\co_2}(0,2;X)>\frac{\left(2-\eps\right)X-4-\sqrt{6X-2}}{\sqrt{24X+1}+1},$$ implying the desired result.
\end{proof}

We adapt the strategy of Mirsky in~\cite{mirsky} to prove Theorem~\ref{loglog}.
\begin{proof}[Proof of Theorem~\ref{loglog}]
 Let $E_{\ce}^*(r,a;X)=\#\{n\le X: \ce(n)\not\equiv r\pmod{a}\}$, and define
$E_{\co_1}^*(r,a;X)$ and $E_{\co_2}^*(r,a;X)$ similarly. Fix $r$; then, in the case of $\CE$, we claim that $E_{\ce}^*(r,a;X)>\log_2\log_3 X-C$ for some constant $C$. 

Recall that $(q)_{\infty}=\prod\limits_{n=1}^{\infty}(1-q^n).$ Then we have by Lemma~\ref{expansions}, considering the coefficient of $q^{\ell(3\ell+1)/2+2}$ in $(q)_{\infty}\CE(q)$, that
\begin{align*}
\ce(\ell(3\ell+1)/2+2)+\sum_{k=1}^\ell(-1)^k\ce\left(\frac{\ell(3\ell+1)}2+2-\frac{k(3k-1)}2\right)\\
+\sum_{k=1}^{\ell-1}(-1)^k\ce\left(\frac{\ell(3\ell+1)}2+2-\frac{k(3k+1)}2\right)+(-1)^{\ell}\ce(2)=0
\end{align*}
for all $\ell\ge1$. For $\ell=2m-1$ this becomes
\begin{align}
\label{cong}
\ce((2m-1)(3m-1)+2)+\sum_{k=1}^{2m-1}(-1)^k\ce((2m-1)(3m-1)+2-k(3k-1)/2)\\
+\sum_{k=1}^{2m-2}(-1)^k\ce((2m-1)(3m-1)+2-k(3k+1)/2)=2.
\end{align}
We claim that some element of $\ce(2m+1),\ce(2m+2),\dotsc,\ce((2m-1)(3m-1)+2)$ is not congruent to $r\pmod{a}$. Suppose otherwise; then the above equation gives
\[
0\equiv r+\sum_{k=1}^{2m-1}(-1)^kr+\sum_{k=1}^{2m-2}(-1)^kr\equiv2\spmod{a},
\]
which is a contradiction. We would like to construct a sequence $m_j$ so that the terms appearing in Equation~\ref{cong} for $m_j$ and $m_i$ do not overlap for $i\neq j$, hence giving a value of $\ce$ not congruent to $r\pmod{a}$ for each $j$. Since the lowest-indexed term in Equation~\ref{cong} for $m_j$ is $2m_j+1$ and the highest is $(2m_j-1)(3m_j-1)+2$, it suffices to set $2m_j+1>(2m_{j-1}-1)(3m_{j-1}-1)+2$, or $2m_j>6m_{j-1}^2-5m_{j-1}+2$. Hence we can choose $m_1=1$, $m_j=3m_{j-1}^2$, so that $m_j=3^{2^{j-1}-1}$. This gives $E_{\ce}^*(r,a;3^{2^{j-1}-1})\ge j$ for all $j$. Setting $j=\lfloor\log_2\log_3X\rfloor$ gives $E_{\ce}^*(r,a;X)\ge\log_2\log_3X-C$, as desired.

Since $E_{\ce}^*(0,a;X)=\sum\limits_{r'\neq 0}E_{\ce}(r',a;X)$, we can write 
\[
\sum_{r'\neq 0}E_{\ce}(r',a;X)>\log_2\log_3X-C,
\] 
from which there is some $r_1$ so that 
\[
E_{\ce}(r_1,a;X)>\frac1{a-1}\log_2\log_3X-C>\frac1a\log_2\log_3X
\]
for $X$ sufficiently large. Then we also have $\sum_{r'\neq r_1}E_{\ce}(r',a;X)= E_{\ce}^*(r_1,a;X)>\log_2\log_3X-C$, giving some $r_2\neq r_1$ so that $E_{\ce}(r_2,a;X)>\frac1a\log_2\log_3X$. This gives us two residues, $r_1$ and $r_2$, with the desired bound in the $\CE$ case.

For $\CO_1$ and $\CO_2$, the same argument applies almost verbatim, with the remark that we use the coefficient of $q^{(2m-1)(3m-1)}$ in $(q)_{\infty}\CO_1(q),(q)_{\infty}\CO_2(q)$ respectively, rather than $q^{(2m-1)(3m-1)+2}$. In both cases we can guarantee that this coefficient must equal $0$. In the case $\CO_1$ the difference between $(2m-1)(3m-1)$ and the nearest exponent of a nonzero coefficient is $3m-1$; in the case $\CO_2$ it is $m$. Noting that $\co_1(0)=\co_2(0)=1\neq0$, the rest of the argument follows.
%In the case that $\ell$ is even, setting $\ell=2m$ yields
%\[
%\ce(m(6m-1))+\sum_{k=1}^{2m-1}(-1)^k\ce(m(6m-1)-k(3k-1)/2)+\sum_{k=1}^{2m-1}(-1)^k\ce(m(6m-1)-k(3k+1)/2)=-2.
%\]
%Suppose there is some $r$ for which $\ce(n)\equiv r\spmod{a}$ for all 
\end{proof}

\end{document}